\documentclass[11pt,a4paper,reqno]{amsart}
\usepackage{amsfonts}
\usepackage{amsmath,amssymb,amsthm,amsxtra}
\usepackage{float}
\usepackage[colorlinks, linkcolor=blue,anchorcolor=Periwinkle,
    citecolor=red,urlcolor=Emerald]{hyperref}
\usepackage[usenames,dvipsnames,svgnames,table]{xcolor}
\usepackage{enumitem}
\usepackage{geometry,array} 
\usepackage{graphicx}
\usepackage{subfigure}
\usepackage{bookmark}
\usepackage{tikz}\usetikzlibrary{matrix}
\usepackage{url}
\usepackage{mathtools}

\usetikzlibrary{decorations.markings}
\tikzset{->-/.style={decoration={  markings,  mark=at position #1 with
    {\arrow{>}}},postaction={decorate}}}
\tikzset{-<-/.style={decoration={  markings,  mark=at position #1 with
    {\arrow{<}}},postaction={decorate}}}

\usepackage[all]{xy}
\usetikzlibrary{arrows}
\usetikzlibrary{decorations.pathreplacing,decorations.pathmorphing,shadings,fadings,calc}
\usepackage[notref,notcite]{showkeys}       

\newcommand{\Note}[1]{\textcolor{NavyBlue}{\texttt{#1}}}
\newcommand{\new}[1]{\textcolor{Blue}{\texttt{#1}}}

\theoremstyle{plain}
\newtheorem{theorem}{Theorem}[section]
\newtheorem{thmx}{Theorem}

\newtheorem{lemma}[theorem]{Lemma}
\newtheorem{corollary}[theorem]{Corollary}
\newtheorem{proposition}[theorem]{Proposition}

\theoremstyle{definition}
\newtheorem{definition}[theorem]{Definition}

\newtheorem{construction}[theorem]{Construction}
\numberwithin{equation}{section}

\def\hua{\mathcal}
\def\kong{\mathbb}
\def\<{\langle}
\def\>{\rangle}

\def\Aut{\operatorname{Aut}}

\def\Sim{\operatorname{Sim}}
\def\Hom{\operatorname{Hom}}

\def\Ext{\operatorname{Ext}}

\def\Stab{\operatorname{Stab}}
\def\Stap{\operatorname{Stab}^\circ}
\def\diff{\operatorname{d}}

\newcommand{\h}{\operatorname{\hua{H}}}            

\renewcommand{\k}{\mathbf{k}}

\newcommand{\Cone}{\operatorname{Cone}}

\def\numbers{\begin{enumerate}[label=\arabic*{$^\circ$}.]}
\def\ends{\end{enumerate}}

\newcommand{\EG}{\operatorname{EG}}       
\newcommand{\EGp}{\operatorname{EG}^\circ}       

\newcommand{\D}{\operatorname{\hua{D}}}
\newcommand{\E}{\operatorname{\hua{E}}}

\newcommand{\per}{\operatorname{per}}
\newcommand\Sph{\operatorname{Sph}}
\def\zero{\hua{H}_\Gamma}
\newcommand{\Tri}{\bigtriangleup}

\def\arrow{red}
\def\surf{\mathbf{S}}                       

\newcommand{\ST}{\operatorname{ST}}        
\newcommand{\BT}{\operatorname{BT}}        

\def\TT{\mathbf{T}}
\def\T{\kong{T}}

\def\M{\mathbf{M}}

\def\ee{\operatorname{\mathfrak{E}}}

\def\surfo{{\mathbf{S}}_\Tri}
\def\cA{\operatorname{CA}}

\def\Rhom{\operatorname{\hua{H}\it{om}}}

\title[Decorated marked surfaces III]
{Decorated marked surfaces III: The derived category of a decorated marked surface}
\author{Aslak Bakke Buan}
\address{Institutt for matematiske fag, NTNU,
 N-7491 Trondheim, Norway.}
\email{aslakb@math.ntnu.no}

\author{Yu Qiu}
\address{Department of Mathematics,
Chinese University of Hong Kong,
N.T., Hong Kong}
\email{yu.qiu@bath.edu}

\author{Yu Zhou}
\address{Yau Mathematical Sciences Center, Tsinghua University, 100084 Beijing, China}
\email{yuzhoumath@gmail.com}

\date{\today}
\begin{document}

\begin{abstract}

    We study the Ginzburg dg algebra $\Gamma_\TT$ associated to the quiver with potential
    arising from a triangulation $\TT$ of a decorated marked surface $\surfo$, in the sense of \cite{QQ}.
    We show that there is a canonical way to identify all finite dimensional derived categories
    $\D_{fd}(\Gamma_\TT)$, denoted by $\D_{fd}(\surfo)$.
    As an application, we show that the spherical twist group $\ST(\surfo)$ associated to
    $\D_{fd}(\surfo)$ acts faithfully on its space of stability conditions.

\end{abstract}

\keywords{Calabi-Yau categories, spherical twists, quivers with potential, derived equivalences, stability conditions}

\maketitle

\section{Introduction}

Cluster algebras and quiver mutation were introduced by Fomin and Zelevinsky \cite{FZ}, and (additive) categorification
of such structures, often in terms of triangulated categories, have successfully contributed to the development of a rich theory,
see e.g. the surveys by Keller\cite{Kel3,Kel4} or Reiten \cite{R}.
Derksen-Weyman-Zelevinsky \cite{DWZ} introduced quivers with potential (QP) and the corresponding Jacobian algebras, and studied mutation of quivers with potential.
Keller-Yang \cite{KY} studied the categorification of such mutations via Ginzburg dg algebras \cite{G}.
One of the applications of their categorification is motivic Donaldson-Thomas invariants, using quantum cluster algebras \cite{K11}.

Additive categorification is deeply related to classical tilting theory \cite{BMRRT,A}.
Algebras related by tilting are derived equivalent, while (Jacobian) algebras related by mutation of quivers with potential are in general not.
However, Keller-Yang constructed in \cite{KY} an equivalence between the derived category $\D(\Gamma(Q,W))$ of the Ginzburg
dg algebra $\Gamma(Q,W)$
of a QP $(Q,W)$ and
that of $\D(\Gamma(\widetilde{Q},\widetilde{W}))$ of a QP $(\widetilde{Q},\widetilde{W})$, obtained by a single mutation from $(Q,W)$.
The equivalence also restricts to the subcategories of dg modules of finite dimensional homology $\D_{fd}(\Gamma(Q,W))$
and $\D_{fd}(\Gamma(\widetilde{Q},\widetilde{W}))$. Note that these subcategories are 3-Calabi-Yau, by \cite{K8}.
However, in general there is no canonical choice for such equivalences, basically because mutation of QPs is only well-defined up to a non-canonical {\em choice of decomposition} of a QP into
a {\em trivial part} and a {\em reduced part} (see Section~\ref{subsec:KY}).

We consider a special class of quivers with potential, that is, those arising from (unpunctured) marked surfaces $\surf$. This class of examples was
first introduced in cluster theory by Fomin-Shapiro-Thurston \cite{FST} and further studied by many authors, including Labardini-Fragoso \cite{LF} who gave the interpretation in terms of corresponding
QPs.
When studying the 3-Calabi-Yau categories and stability conditions,
it is natural to decorate the marked surface $\surf$
with a set $\Tri$ of decorating points (which are zeroes of the
corresponding quadratic differentials, cf. \cite{BS,QQ}).
More details about motivation and background can be found in \cite{QQ}.

Building on the prequels \cite{QQ,QZ2},
we prove a class of intrinsic derived equivalences that
are compatible with Keller-Yang's and are stronger in this special case.
More precisely, this class of equivalences implies the following main result.

\begin{thmx}[see Theorem~\ref{thm:comp}]\label{thma}
There is a unique canonical 3-Calabi-Yau category $\D_{fd}(\surfo)$ associated to a decorated marked surface $\surfo$.
\end{thmx}

Given a Ginzburg dg algebra $\Gamma$, one can
consider the {\em spherical twist group} $\ST(\Gamma)$ of $\D_{fd}(\Gamma)$ in $\Aut\D_{fd}(\Gamma)$.
In particular for a decorated marked surface $\surfo$, we study the spherical twist group  $\ST(\surfo)$ and the principal component $\Stap(\surfo)$ of the space of stability conditions on $\D_{fd}(\surfo)$ (see Section 5 for details).
We then obtain the following, as an application of our main theorem.

\begin{thmx}[Theorem~\ref{thmbb}]\label{thmb}
The spherical twist group $\ST(\surfo)$ acts faithfully on the principal component $\Stap(\surfo)$ of the space of stability conditions on $\D_{fd}(\surfo)$.
\end{thmx}

We give preliminary results and background in Section 2, we give an explicit description of Keller-Yang's equivalence on the finite derived category in Section 3, we prove our main result Theorem~\ref{thma} in Section 4, and we give background for and proof of Theorem~\ref{thmb}
in Section 5.

Throughout the paper, a composition $fg$ of morphisms $f$ and $g$ means first $g$ and then $f$. But a composition $ab$ of arrows $a$ and $b$ means first $a$ then $b$. Any (dg) module is a right (dg) module.

\subsection*{Acknowledgements}
The second author would like to thank A.~King and J.~Grant for
interesting discussions.
This work was supported by the Research Council of Norway, grant No.NFR:231000.

\section{Preliminaries}\label{sec:bg}
\subsection{Decorated marked surfaces}
Throughout the paper,
$\surf$ denotes a \emph{marked surface} without punctures in the sense of Fomin-Shapiro-Thurston \cite{FST}.
That is, $\surf$ is a connected compact surface with a fixed orientation and
with a finite set $\M$ of marked points on the (non-empty) boundary $\partial\surf$
having the property that
each connected component of $\partial\surf$ contains at least one marked point.
Up to homeomorphism, $\surf$ is determined by the following data\new{:}
\begin{itemize}
\item the genus $g$;
\item the number $|\partial\surf|$ of boundary components;
\item the integer partition of $|\M|$ into $|\partial\surf|$
parts describing the number of marked points
on each boundary component.
\end{itemize}
We require that
\begin{gather}\label{eq:n}
n=6g+3|\partial\surf|+|\M|-6
\end{gather}
is at least one. A triangulation of $\surf$ is a maximal collection of non-crossing and non-homotopic simple curves on $\surf$, whose endpoints are in $\M$. It is well-known that any triangulation of $\surf$ consists of $n$
simple curves (\cite[Proposition~2.10]{FST}) and divides $\surf$ into
\begin{gather}\label{eq:Tri}
    \aleph=\frac{2n+|\M|}{3}
\end{gather}
triangles (\cite[(2.9)]{QQ}).

\begin{definition}[{\cite[Definition~3.1]{QQ}}]\label{def:arcs}
A \emph{decorated marked surface} $\surfo$ is a marked surface $\surf$ together with
a fixed set $\Tri$ of $\aleph$ `decorating' points in the interior of $\surf$
(where $\aleph$ is defined in \eqref{eq:Tri}),
which serve as punctures.
Moreover, a (simple) \emph{open arc} in $\surfo$ is (the isotopy class of) a (simple) curve in $\surfo-\Tri$
that connects two
marked points in $\M$, which is neither isotopic to a boundary segment nor to a point.
\end{definition}

A triangulation of $\TT$ of $\surfo$ is a collection of 
simple open
arcs that divides $\surfo$ into $\aleph$ triangles, each containing exactly one decorating point
inside (cf. \cite[\S~3]{QQ}).
We also have the notion of (forward/backward) flips of triangulations of $\surfo$,
cf. Figure~\ref{fig:flips}.
Denote by $\EG(\surfo)$ the exchange graph of triangulations of $\surfo$,
that is the oriented graph whose vertices are the triangulations and whose edges are
the forward flips between them.

From now on, we will fix a connected component $\EGp(\surfo)$.
When we say a triangulation of $\surfo$ later, we always mean that it is in this component.

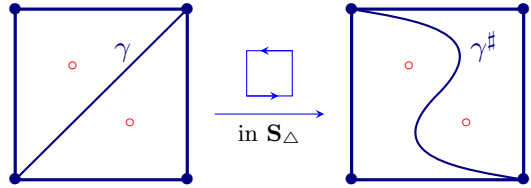
\begin{figure}[ht]\centering
\begin{tikzpicture}[scale=.4]
    \path (-135:4) coordinate (v1)
          (-45:4) coordinate (v2)
          (45:4) coordinate (v3);
\draw[NavyBlue,very thick](v1)to(v2)node{$\bullet$}to(v3);
    \path (-135:4) coordinate (v1)
          (45:4) coordinate (v2)
          (135:4) coordinate (v3);
\draw[NavyBlue,very thick](v2)node{$\bullet$}to(v3)node{$\bullet$}to(v1)node{$\bullet$}
(45:1)node[above]{$\gamma$};
\draw[>=stealth,NavyBlue,thick](-135:4)to(45:4);
\draw[red,thick](135:1.333)node{\tiny{$\circ$}}(-45:1.333)node{\tiny{$\circ$}};
\end{tikzpicture}
\begin{tikzpicture}[scale=1.2, rotate=180]
\draw[blue,<-,>=stealth](3-.6,.7)to(3+.6,.7);
\draw(3,.7)node[below,black]{\footnotesize{in $\surfo$}};
\draw[blue](3-.25,.5-.5)rectangle(3+.25,.5);\draw(3,1.5)node{};
\draw[blue,->,>=stealth](3-.25,.5-.5)to(3+.1,.5-.5);
\draw[blue,->,>=stealth](3+.25,.5)to(3-.1,.5);
\end{tikzpicture}
\begin{tikzpicture}[scale=.4];
    \path (-135:4) coordinate (v1)
          (-45:4) coordinate (v2)
          (45:4) coordinate (v3);
\draw[NavyBlue,very thick](v1)to(v2)node{$\bullet$}to(v3)
(45:1)node[above right]{$\gamma^\sharp$};
    \path (-135:4) coordinate (v1)
          (45:4) coordinate (v2)
          (135:4) coordinate (v3);
\draw[NavyBlue,very thick](v2)node{$\bullet$}to(v3)node{$\bullet$}to(v1)node{$\bullet$};
\draw[>=stealth,NavyBlue,thick](135:4).. controls +(-10:2) and +(45:3) ..(0,0)
                             .. controls +(-135:3) and +(170:2) ..(-45:4);
\draw[red,thick](135:1.333)node{\tiny{$\circ$}}(-45:1.333)node{\tiny{$\circ$}};
\end{tikzpicture}
\caption{A forward flip}
\label{fig:flips}
\end{figure}

\subsection{Quivers with potential and Ginzburg dg algebras}\label{sec:QP}
Let $Q$ be a quiver without loops or oriented 2-cycles. A potential $W$ is a linear combination of cycles in $Q$. Denote by $Q_0$ the set of vertices of $Q$ and by $Q_1$ the set of arrows of $Q$. Denote by $s(a)$ (resp. $t(a)$) the source (resp. target) of an arrow $a$. Denote by $e_i$ the trivial path at a vertex $i\in Q_0$.


Fix an algebraically closed field $\k$. All categories considered are
$\k$-linear.
Denote by $\Gamma=\Gamma(Q,W)$ the \emph{Ginzburg dg algebra (of degree 3)} associated to
a quiver with potential $(Q,W)$,
which is constructed as follows (cf. \cite{G,KY}):
\begin{itemize}
\item   Let $\overline{Q}$ be the graded quiver whose vertex set is $Q_0$
and whose arrows are:
\begin{itemize}
\item the arrows in $Q_1$ with degree $0$;
\item an arrow $a^*:j\to i$ with degree $-1$ for each arrow $a:i\to j$ in $Q_1$;
\item a loop $e_i^*:i\to i$ with degree $-2$ for each vertex $i$ in $Q_0$.
\end{itemize}
The underlying graded algebra of $\Gamma$ is the completion of
the graded path algebra $\k \overline{Q}$ in the category of graded vector spaces
with respect to the ideal generated by the arrows of $\overline{Q}$.
\item The differential $\diff$ of $\Gamma$ is the unique continuous linear endomorphism,
homogeneous of degree $1$, which satisfies the Leibniz rule and takes the following values
\begin{itemize}
  \item $\diff a = 0$ for any $a\in Q_1$,
  \item $\diff a^* = \partial_a W$ for any $a\in Q_1$ and
  \item $\diff \sum_{i\in Q_0} e_i^*=\sum_{a\in Q_1} \, [a,a^*]$.
\end{itemize}
\end{itemize}
Denote by $\D(\Gamma)$ the derived category of $\Gamma$. We will focus on studying the finite-dimensional derived category
$\D_{fd}(\Gamma)$ of $\Gamma$, which is the full subcategory of $\D(\Gamma)$ consisting of the dg $\Gamma$-modules whose total homology is finite dimensional. This category is 3-Calabi-Yau \cite{K8}, that is, for any pair of objects $L,M$ in $\D_{fd}(\Gamma)$, we have a natural isomorphism
\begin{equation}
    \Hom_{\D_{fd}(\Gamma)}(L,M)\cong D\Hom_{\D_{fd}(\Gamma)}(M,L[3])
\end{equation}
where $D=\Hom_\k(-,\k)$.

Following \cite{FST,LF}, one can associate a quiver with potential $(Q_\TT,W_\TT)$ to
each triangulation $\TT$ of $\surfo$ as follows:
\begin{itemize}
  \item the vertices of $Q_\TT$ are indexed by the open arcs in $\TT$;
  \item each clockwise angle in a triangle of $\TT$ gives an arrow between the vertices indexed by the edges of the angle;
  \item each triangle in $\TT$ with three edges being open arcs gives a 3-cycle (up to cyclic permutation) and the potential $W_\TT$ is the sum of such 3-cycles.
\end{itemize}
Then we have the corresponding Ginzburg dg algebra $\Gamma_\TT=\Gamma(Q_\TT,W_\TT)$
and the 3-Calabi-Yau category $\D_{fd}(\Gamma_\TT)$.

\subsection{Mutations and Keller-Yang's equivalences}\label{subsec:KY}

Let $(Q,W)$ be a quiver with potential. For a vertex $k$ in $Q$, the \emph{pre-mutation}  $\widetilde{\mu}_k(Q,W)=(\widetilde{Q},\widetilde{W})$ at $k$ is a new quiver with potential, defined as follows. The new quiver $\widetilde{Q}$ is obtained from $Q$ by
\begin{enumerate}
\item[Step 1] For any composition $ab$ of two arrows with $t(a)=s(b)=k$, add a new arrow $[ab]$ from $s(a)$ to $t(b)$.
\item[Step 2] Replace each arrow $a$ with $s(a)=k$ or $t(a)=k$ by an arrow $a'$ with $s(a')=t(a)$ and $t(a')=s(a)$.
\end{enumerate}
The new potential
\[\widetilde{W}=\widetilde{W}_1+\widetilde{W}_2\]
where $\widetilde{W}_1$ is obtained from $W$ by replacing each composition $ab$ of arrows with $t(a)=s(b)=k$ by $[ab]$, and $\widetilde{W}_2$ is the sum of the 3-cycles of the form $[ab]b'a'$. Denote by $\widetilde{\Gamma}=\Gamma(\widetilde{Q},\widetilde{W})$ the corresponding Ginzburg dg algebra.
Let $P_i=e_i\Gamma$ be the indecomposable direct summand of $\Gamma$ corresponding to a vertex $i$. Denote by $P_i^?$ a copy of $P_i$, where $?$ can be an arrow, or a pair of arrows.

The \emph{forward mutation} of $\Gamma$ at $P_k$ in $\per\Gamma$ is $\mu^\sharp_k(\Gamma)=\bigoplus_{i\in Q_0} \widetilde{P}_i$, where $\widetilde{P}_i=P_i$ if $i\neq k$, and $\widetilde{P}_k$ has the underlying graded space
\[|\widetilde{P}_k|=P_k[1]\oplus\bigoplus_{\rho\in Q_1:t(\rho)=k}P^\rho_{s(\rho)}\]
with the differential
\[d_{\widetilde{P}_k}=\begin{pmatrix}
d_{P_k[1]}&0\\
\rho & d_{P^\rho_{s(\rho)}}
\end{pmatrix}.\]

\begin{construction}[\cite{KY}]\label{con:KY}
There is a map between dg algebras \[f_{?}:\widetilde{\Gamma}\to\Rhom_\Gamma(\mu^\sharp_k(\Gamma),\mu^\sharp_k(\Gamma))\]
constructed as follows, where $\xrightarrow{a}$ means the left multiplication by $a$:
\begin{enumerate}
	\item for an arrow $\alpha\in Q_1$ with $t(\alpha)=k$,
	\begin{itemize}
		\item $f_{\alpha'}:P_{s(\alpha)}\to \widetilde{P}_k$ of degree 0 is given by
		\[P_{s(\alpha)}\xrightarrow{\begin{pmatrix}
			0 \\ \delta_{\alpha,\rho}
			\end{pmatrix}} P_k[1]\oplus\bigoplus_{\rho\in Q_1:t(\rho)=k}P^\rho_{s(\rho)} \]
		where $\delta_{\alpha,\rho}=1$ if $\alpha=\rho$ and 0 else;
		\item $f_{\alpha'^\ast}:\widetilde{P}_k\to P_{s(\alpha)}$ of degree -1 is given by
		\[P_k[1]\oplus\bigoplus_{\rho\in Q_1:t(\rho)=k}P^\rho_{s(\rho)}\xrightarrow{\begin{pmatrix}
			-\alpha e_k^\ast & -\alpha\rho^\ast
			\end{pmatrix}} P_{s(\alpha)}\]
	\end{itemize}
	\item for an arrow $\beta\in Q_1$ with $s(\beta)=k$,
	\begin{itemize}
		\item $f_{\beta'}:\widetilde{P}_k\to P_{t(\beta)}$ of degree 0 is given by
		\[P_k[1]\oplus\bigoplus_{\rho\in Q_1:t(\rho)=k}P^\rho_{s(\rho)}\xrightarrow{\begin{pmatrix}
			\beta^\ast & \partial_{\rho\beta}W
			\end{pmatrix}}P_{t(\beta)}\]
		\item $f_{\beta'^\ast}:P_{t(\beta)}\to \widetilde{P}_k$ of degree -1 is given by
		\[P_{t(\beta)}\xrightarrow{\begin{pmatrix}
			-\beta \\ 0
			\end{pmatrix}}P_k[1]\oplus\bigoplus_{\rho\in Q_1:t(\rho)=k}P^\rho_{s(\rho)}\]
	\end{itemize}
	\item for a pair of arrows $\alpha,\beta\in Q_1$ with $t(\alpha)=k=s(\beta)$, \[f_{[\alpha\beta]}:P_{t(\beta)}\xrightarrow{-\alpha\beta}P_{s(\alpha)}\]
	and
	\[f_{[\alpha\beta]^\ast}:P_{s(\alpha)}\xrightarrow{0}P_{t(\beta)}\]
	\item for an arrow $\gamma$ in $Q_1$ not incident to $k$, $f_{\gamma}:P_{t(\gamma)}\xrightarrow{\gamma} P_{s(\gamma)}$ and $f_{\gamma^\ast}:P_{s(\gamma)}\xrightarrow{\gamma^\ast} P_{t(\gamma)}$;
	\item for a vertex $i\in Q_0$ different from $k$, $f_{{e'^\ast}_i}: P_i\xrightarrow{e^\ast_i} P_i$;
	\item $f_{e'^\ast_k}:\widetilde{P}_k\to \widetilde{P}_k$ of degree -2 is given by
	\[P_k[1]\oplus\bigoplus_{\rho\in Q_1:t(\rho)=k}P^\rho_{s(\rho)}
	\xrightarrow{\begin{pmatrix}
		-e^\ast_k & -\rho^\ast \\ 0 & 0
		\end{pmatrix}}
	P_k[1]\oplus\bigoplus_{\rho\in Q_1:t(\rho)=k}P^\rho_{s(\rho)} \]
\end{enumerate}
\end{construction}

The main result in \cite{KY} is the following derived equivalence.

\begin{theorem}[{\cite[Proposition~3.5 and Theorem~3.2]{KY}}]\label{thm:KY}
The map $f_?$ is a homomorphism of dg algebras. In this way, $\mu^\sharp_k(\Gamma)$ becomes a left dg $\widetilde{\Gamma}$-module. Moreover, the $\widetilde{\Gamma}$-$\Gamma$-bimodules $\mu^\sharp_k(\Gamma)$ induces a triangle equivalence $F=?\overset{L}{\otimes}_{\widetilde{\Gamma}}\mu^\sharp_k(\Gamma): \D(\widetilde{\Gamma})\to\D(\Gamma)$, with inverse $\mathcal{H}om_{\Gamma}(\mu^\sharp_k(\Gamma),?):\D(\Gamma)\to\D(\widetilde{\Gamma})$.
\end{theorem}


Introduced in \cite{DWZ}, the \emph{mutation} $\mu_k(Q,W)$ of $(Q,W)$ at $k$ is obtained from $(\widetilde{Q},\widetilde{W})$ by taking its reduced part $(\widetilde{Q}_{\mbox{red}},\widetilde{W}_{\mbox{red}})$. That is, there is a right-equivalence between $(\widetilde{Q},\widetilde{W})$ and the direct sum of quivers with potential $(\widetilde{Q}_{\mbox{triv}},\widetilde{W}_{\mbox{triv}})\oplus(\widetilde{Q}_{\mbox{red}},\widetilde{W}_{\mbox{red}})$ such that $(\widetilde{Q}_{\mbox{triv}},\widetilde{W}_{\mbox{triv}})$ is trivial (in the sense that its Jacobian algebra is the path algebra of the vertices) and $(\widetilde{Q}_{\mbox{red}},\widetilde{W}_{\mbox{red}})$ is reduced (in the sense that $\widetilde{W}_{\mbox{red}}$ contains no 2-cycles). Here, the direct sum of two quivers with potential is a quiver with potential, whose quiver is the union of arrows in the two quivers and whose potential is the sum of the two potentials. A right equivalence between two quivers with potential is a homomorphism of path algebras of the quivers, which sends the first potential to the second.

In general, the choice of such right-equivalence is not unique. However, for the quiver with potential $(Q_\TT,W_\TT)$ associated to a triangulation $\TT$ of a decorated surface, the right-equivalence can be the identity, which is a canonical choice. This is because any 2-cycle in the potential $\widetilde{W}$ of the pre-mutation $(\widetilde{Q},\widetilde{W})=\widetilde{\mu}_k(Q_\TT,W_\TT)$ contains no common arrows with any other terms in $\widetilde{W}$ (see Case 1 in the proof of \cite[Theorem~30]{LF}). So one can remove all of the 2-cycles from $\widetilde{W}$ and remove the arrows in these 2-cycles from $\widetilde{Q}$ to get the reduced part. This means that the mutation $\mu_k(Q_\TT,W_\TT)$ is a direct summand of the pre-mutation $\widetilde{\mu}_k(Q_\TT,W_\TT)$. Then there is a canonical quasi-isomorphism between $\widetilde{\Gamma}$ and $\Gamma(\mu_k(Q_\TT,W_\TT))$.

Moreover, by \cite[Theorem~30]{LF}, $\mu_k(Q_\TT,W_\TT)$ is the same as $(Q_{\TT'},W_{\TT'})$, where $\TT'=f^\sharp_k(\TT)$ is the forward flip of $\TT$ w.r.t. $k$.
Then we have a canonical quasi-isomorphism between $\widetilde{\Gamma}$ and $\Gamma_{\TT'}$, which makes $\mu^\sharp_k(\Gamma_\TT)$ a $\Gamma_{\TT'}$-$\Gamma_{\TT}$-bimodules. By Theorem~\ref{thm:KY}, we have the following notion.


\begin{definition}[Keller-Yang's equivalence]
Using the above notation, we call the triangle equivalence
\[\kappa_{\TT'}^{\TT}:=?\overset{L}{\otimes}_{\Gamma_{\TT'}}\mu^\sharp_k(\Gamma_\TT):\D(\Gamma_{\TT'})\to \D(\Gamma_\TT).\]
the \emph{Keller-Yang's equivalence} from $\TT$ to $\TT'$.
\end{definition}



\section{Keller-Yang's equivalences on finite-dimensional derived categories}\label{sec:KYsurf}

\subsection{Hearts and spherical objects}\label{sec:bi}
A \emph{bounded t-structure} \cite{BBD}
on a triangulated category $\hua{D}$ is
a full subcategory $\hua{P} \subset \hua{D}$
with $\hua{P}[1] \subset \hua{P}$ such that
\begin{itemize}
	\item if one defines
	\[
	\hua{P}^{\perp}=\{ G\in\hua{D}\mid \Hom_{\hua{D}}(F,G)=0,
	\forall F\in\hua{P}  \},
	\]
	then, for every object $E\in\hua{D}$, there is
	a (unique) triangle $F \to E \to G\to F[1]$ in $\hua{D}$
	with $F\in\hua{P}$ and $G\in\hua{P}^{\perp}$.
	\item
	for every object $M$,
	the shifts $M[k]$ are in $\hua{P}$ for $k\gg0$ and in $\hua{P}^{\perp}$ for $k\ll0$.
\end{itemize}
The \emph{heart} of a bounded t-structure $\hua{P}$ is the full subcategory
\[
\h=  \hua{P}^\perp[1]\cap\hua{P}
\]
and any bounded t-structure is determined by its heart.

Note that any heart of a triangulated category is abelian \cite{BBD}.
Let $(\hua{T},\hua{F})$ be a torsion pair in a heart $\h$, that is, $\Hom_{\h}(\hua{T},\hua{F})=0$ and for any object $X\in\h$, there exists a short exact sequence $0\to T\to X\to F\to 0$ with $T\in\hua{T}$ and $F\in\hua{F}$.
Then there are hearts $\h^\sharp$ and $\h^\flat$, called
forward/backward tiltings of $\h$ with respect to this torsion pair
(in the sense of Happel-Reiten-Smal\o~\cite{HRS}).
In particular, the forward (resp. backward) tilting is simple if
$\hua{F}$ (resp. $\hua{T}$) is generated by a single rigid simple in $\h$.
See \cite[\S~3]{KQ} for details.
The \emph{exchange graph} $\EG(\D)$ of a triangulated category $\hua{D}$ is the oriented graph whose vertices are all hearts in $\hua{D}$ and whose edges correspond
to simple forward tiltings between them.

Let $\TT$ be a triangulation in $\EGp(\surfo)$. Denote by $\EGp(\Gamma_\TT)$ the principal component of
the exchange graph $\EG(\D_{fd}(\Gamma_\TT))$, that is the connected component containing the canonical heart $\h_\TT$.
Denote by
\begin{gather}\label{eq:sph}
\Sph(\Gamma_\TT)=\bigcup_{\h\in\EGp(\Gamma_\TT)}\Sim\h,
\end{gather}
the set of reachable spherical objects (cf. Definition~\ref{def:sph}), where $\Sim\h$ is the set of simple objects in $\h$.
By \cite[Proposition~3.2 and (3.3)]{Q}, there is an isomorphism of oriented graphs
\begin{equation}\label{eq:cong}
\EGp(\surfo)\cong\EGp(\Gamma_\TT)
\end{equation}
which sends $\TT$ to the canonical heart $\h_\TT$. We denote by $\h_\TT^{\TT'}$ the heart corresponding to $\TT'\in\EGp(\surfo)$.

\subsection{Koszul duality}\label{subsec:koszul}

Let $\Gamma$ be the Ginzburg dg algebra associated to a quiver with potential $(Q,W)$. Let $\h$ be a heart obtained from the canonical heart by a sequence of simple tiltings. Denote by $S$ the direct sum of non-isomorphic simples in $\h$. Consider the dg endomorphism algebra
\begin{gather}\label{eq:REnd}
    \ee(S)=\Rhom_{\Gamma}(S, S).
\end{gather}
Since $S$ generates $\D_{fd}(\Gamma)$ (by taking extensions, shifts in both directions and direct summands), by \cite{Kel} (cf. also \cite[Section~8]{Kel2}), we have the following triangle equivalence:
\begin{gather}\label{eq:DE}
\xymatrix@C=4pc{
    \D_{fd}(\Gamma) \ar[rr]^{ \Rhom_{\Gamma}(S, ?) }
          && \per\ee(S),
}\end{gather}

The homology of $\ee(S)$ is the Ext-algebra \[\E(\h):=\Ext^{\mathbb{Z}}_{\D_{fd}(\Gamma)}(S,S)=\bigoplus_{n\in\mathbb{Z}}\Hom_{\D_{fd}(\Gamma)}(S,S[n]).\]
In general, one needs to consider a certain $A_\infty$-structure on $\E(\h)$ (which is induced from the potential $W$, see \cite[Appendix]{K8}) such that it is derived equivalent to $\ee(S)$. However, in the surface case, only ordinary multiplication in the induced $A_\infty$-structure is non-trivial (see \cite[Lemma~A.2]{QZ2}). So we have that for any $\TT,\TT'\in\EGp(\surfo)$, there is a triangle equivalence
\begin{gather}\label{eq:DE2}
\xymatrix@C=4pc{
    \D_{fd}(\Gamma_\TT) \ar[rr]^{ \Ext^\mathbb{Z}_{\Gamma_\T}(S_\TT^{\TT'}, ?)\qquad }
          && \per\E(\h_\TT^{\TT'}),
}\end{gather}
where $S_\TT^{\TT'}$ is the direct sum of non-isomorphic simples in the heart $\h_\TT^{\TT'}$.

\subsection{Keller-Yang's equivalences on simples}

Let $\Gamma$ be the Ginzburg dg algebra associated to a quiver with potential $(Q,W)$. Denote by $S_i$ the simple $\Gamma$-module corresponding to a vertex $i$ of $Q$. There is a short exact sequence of dg $\Gamma$-modules
\[
\xymatrix{0\ar[r] & \ker(\zeta_i)\ar[r] & P_i\ar[r]^{\zeta_i} & S_i\ar[r] & 0,}
\]
where $\zeta_i$ is the canonical projection from $P_i$ to $S_i$ and
\[
\ker(\zeta_i)=\bigoplus_{\alpha:i\to j\in \overline{Q}_1} \alpha P_j
\]
with the induced differential.
Therefore, $S_i$ has a cofibrant resolution (see \cite[Section~2.12]{KY} for definition and properties of this notion) $\mathbf{p}S_i$ with underlying graded vector space
\begin{equation}\label{eq:under}
|\mathbf{p}S_i|=P_i[3]\oplus\bigoplus_{\rho\in Q_1:t(\rho)=i}P^\rho_{s(\rho)}[2]\oplus\bigoplus_{\tau\in Q_1:s(\tau)=i}P^\tau_{t(\tau)}[1]\oplus P_i
\end{equation}
and with the differential
\begin{equation}\label{eq:diff}
d_{\mathbf{p}S_i}=\left(\begin{smallmatrix}
d_{P_i[3]} & 0 & 0 & 0 \\
\rho & d_{P_{s(\rho)}[2]} & 0 & 0\\
-\tau^\ast & -\partial_{\rho\tau}w&d_{P_{t(\tau)}[1]}&0\\
e_i^\ast& \rho^\ast & \tau & d_{P_i}
\end{smallmatrix}\right)
\end{equation}
Note that any morphism from $S_i$ to $S_j$ in $\D_{fd}(\Gamma)$ is induced by a homomorphism of dg $\Gamma$-modules from $\mathbf{p}S_i$ to $S_j$. Hence each arrow in $\overline{Q}_1$ starting at $i$ or the trivial path $e_i$ at $i$ induces a morphism $\pi_\alpha$ in $\D_{fd}(\Gamma)$ starting at $S_i$ as follows
\begin{itemize}
\item $\pi_{e_i}:S_i\to S_i$ is the identity induced by the projection from $P_i$ to $S_i$;
\item $\pi_\tau: S_i\to S_j[1]$ for $\tau:i\to j\in Q_1$ is induced by the projection from $P_{t(\tau)}[1]$ to $S_j[1]$;
\item $\pi_{\rho^\ast}:S_i\to S_j[2]$ for $\rho:j\to i\in Q_1$ is induced by the projection $P_{s(\rho)}[2]$ to $S_j[2]$;
\item $\pi_{e^\ast_i}:S_i\to S_i[3]$ is induced by the projection from $P_i[3]$ to $S_i[3]$.
\end{itemize}

The morphisms $\pi_?$ above can be extended naturally to elements in $\Ext^{\mathbb{Z}}_{\D_{fd}(\Gamma)}(S,S)$. Moreover, they form a basis.


\begin{proposition}\cite[Lemma~2.15 and its proof]{KY}\label{prop:basis}
The morphisms $\pi_\alpha$, where $\alpha$ is a trivial path or an arrow in $\overline{Q}$, form a basis of $\E(\h_\Gamma)$, where $\h_\Gamma$ is the canonical heart.
\end{proposition}

Let $(\widetilde{Q},\widetilde{W})$ be the pre-mutation of $(Q,W)$ at a vertex $k$ and $\widetilde{\Gamma}$ the corresponding Ginzburg dg algebra. Let $F:\D(\widetilde{\Gamma})\to \D(\Gamma)$ be the triangle equivalence given in Theorem~\ref{thm:KY}. Denote by $\widetilde{S}_i$ the simple $\widetilde{\Gamma}$-module corresponding to $i\in \widetilde{Q}_0=Q_0$.

\begin{construction}\label{con:sharp}
We define objects $S_i^\sharp$, $i\in Q_0$, in $\D(\Gamma)$ as follows. For $i\neq k$, define $S^\sharp_i$ by the triangle
\[S_i[-1]\xrightarrow{\pi_\rho[-1]}\bigoplus\limits_{\begin{smallmatrix}
\rho\in Q_1\\ s(\rho)=i\\t(\rho)=k
\end{smallmatrix}}S^{\rho}_{t(\rho)}\to S_i^\sharp\to S_i\]
where $S^\rho_k$ is a copy of $S_k$; for $i=k$, define $S^\sharp_k$ to be $S_k[1]$. Note that for a vertex $j\in Q_0$, if there is no arrow from $j$ to $k$ then $S^\sharp_j=S_j$.
\end{construction}

By replacing $P_j$'s by $\widetilde{P}_j$'s in \eqref{eq:under} and \eqref{eq:diff}, we get the cofibrant resolution $\mathbf{p}\widetilde{S}_i$ of $\widetilde{S}_i$. For $i\neq k$, by Construction~\ref{con:KY} (cf. also the proof of \cite[Lemma 3.12]{KY}), we have that $F(\mathbf{p}\widetilde{S}_i)$ has the underlying graded space
\[\begin{array}{cccccccccc}
&P_i[3]
\oplus\bigoplus\limits_{\begin{smallmatrix}
	\alpha\in Q_1\\s(\alpha)\neq k\\t(\alpha)=i
	\end{smallmatrix}}P^\alpha_{s(\alpha)}[2]
\oplus \bigoplus\limits_{\begin{smallmatrix}
	a,b\in Q_1\\t(a)=s(b)=k\\t(b)=i
	\end{smallmatrix}}P^{a,b}_{s(a)}[2]
\oplus\bigoplus\limits_{\begin{smallmatrix}
	c\in Q_1\\ s(c)=i\\ t(c)=k
	\end{smallmatrix}}P^c_k[3]
\oplus\bigoplus\limits_{\begin{smallmatrix}
	p,q\in Q_1\\ s(p)=i\\t(p)=t(q)=k
	\end{smallmatrix}}P^{p,q}_{s(q)}[2]\\
\oplus&\bigoplus\limits_{\begin{smallmatrix}
	\beta\in Q_1\\s(\beta)=i\\t(\beta)\neq k
	\end{smallmatrix}}P^\beta_{t(\beta)}[1]
\oplus\bigoplus\limits_{\begin{smallmatrix}
	l,g\in Q_1\\ s(l)=i\\ t(l)=s(g)=k
	\end{smallmatrix}}P^{l,g}_{t(g)}[1]
\oplus\bigoplus\limits_{\begin{smallmatrix}
	h\in Q_1\\ s(h)=k\\ t(h)=i
	\end{smallmatrix}}P^h_k[2]
\oplus\bigoplus\limits_{\begin{smallmatrix}
	x,y\in Q_1\\ t(x)=s(y)=k\\ t(y)=i
	\end{smallmatrix}}P^{x,y}_{s(x)}[1]
\oplus P_i
\end{array}
\]
with the differential
\[
\left(\begin{smallmatrix}
d_{P_i[3]}&\\
\alpha&d_{P^\alpha_{s(\alpha)}[2]}\\
ab&0&d_{P^{a,b}_{s(a)}[2]}\\
0&0&0&d_{P^c_k[3]}\\
\delta_{p,q}&0&0&\delta_{c,p}q&d_{P^{p,q}_{s(q)}[2]}\\
-\beta^\ast&-\partial_{\alpha\beta}W&-\partial_{ab\beta}W&0&0&d^{\beta}_{P_{t(\beta)}[1]}\\
0&-\partial_{\alpha lg}W&0&-\delta_{c,l}g^\ast&-\delta_{p,l}\partial_{qg}W&0&d_{P^{l,g}_{t(g)}[1]}\\
-h&0&0&0&0&0&0&d_{P^h_k[2]}\\
0&0&-\delta_{a,x}\delta_{b,y}&0&0&0&0&-\delta_{h,y}x&d_{P^{x,y}_{s(x)}[1]}\\
e_i^\ast&\alpha^\ast&0&-c e_k^\ast&-pq^\ast&\beta&lg&-h^\ast&-\partial_{xy}W&d_{P_i}
\end{smallmatrix}\right)\]

On the other hand, as a dg $\Gamma$-module, $S^\sharp_i$ has a cofibrant resolution $\mathbf{p}S^\sharp_i$ whose underlying graded space is
\[\begin{array}{rl}
&P_i[3]
\oplus\bigoplus\limits_{\begin{smallmatrix}
	\widetilde{\alpha}\in Q_1\\s(\widetilde{\alpha})\neq k\\t(\widetilde{\alpha})=i
	\end{smallmatrix}}P^{\widetilde{\alpha}}_{s(\widetilde{\alpha})}[2]
\oplus\bigoplus\limits_{\begin{smallmatrix}
	\widetilde{h}\in Q_1\\ s(\widetilde{h})=k\\ t(\widetilde{h})=i
	\end{smallmatrix}}P^{\widetilde{h}}_{k}[2]
\oplus\bigoplus\limits_{\begin{smallmatrix}
	\widetilde{\beta}\in Q_1\\s(\widetilde{\beta})=i\\t(\widetilde{\beta})\neq k
	\end{smallmatrix}}P^{\widetilde{\beta}}_{t(\widetilde{\beta})}[1]
\oplus\bigoplus\limits_{\begin{smallmatrix}
	\sigma\in Q_1\\s(\sigma)=i\\t(\sigma)=k
	\end{smallmatrix}}P^\sigma_{k}[1]
\oplus P_i\\
\oplus&\bigoplus\limits_{\begin{smallmatrix}
	\widetilde{c}\in Q_1\\ s(\widetilde{c})=i\\ t(\widetilde{c})=k
	\end{smallmatrix}}P^{\widetilde{c}}_k[3]
\oplus\bigoplus\limits_{\begin{smallmatrix}
	\widetilde{p},\widetilde{q}\in Q_1\\ s(\widetilde{p})=i\\t(\widetilde{p})=t(\widetilde{q})=k
	\end{smallmatrix}}P^{\widetilde{p},\widetilde{q}}_{s(\widetilde{q})}[2]
\oplus\bigoplus\limits_{\begin{smallmatrix}
	\widetilde{l},\widetilde{g}\in Q_1\\ s(\widetilde{l})=i\\ t(\widetilde{l})=s(\widetilde{g})=k
	\end{smallmatrix}}P^{\widetilde{l},\widetilde{g}}_{t(\widetilde{g})}[1]
\oplus\bigoplus\limits_{\begin{smallmatrix}
	\tau\in Q_1\\ s(\tau)=i\\ t(\tau)=k
	\end{smallmatrix}}P^\tau_{k}
\end{array}\]
with the differential
\[\left(\begin{smallmatrix}
d_{P_i[3]}\\
\widetilde{\alpha}&d_{P^{\widetilde{\alpha}}_{s(\widetilde{\alpha})}[2]}\\
\widetilde{h}&0&d_{P^{\widetilde{h}}_{k}[2]}\\
-\widetilde{\beta}^\ast&-\partial_{\widetilde{\alpha}\widetilde{\beta}}W&-\partial_{\widetilde{h}\widetilde{\beta}}W&d_{P^{\widetilde{\beta}}_{t(\widetilde{\beta})}[1]}\\
-\sigma^\ast&-\partial_{\widetilde{\alpha}\sigma}W&0&0&d_{P^\sigma_{k}[1]}\\
e_i^\ast&\widetilde{\alpha}^\ast&\widetilde{h}^\ast&\widetilde{\beta}&\sigma&d_{P_i}\\
0&0&0&0&0&0&d_{P_k^{\widetilde{c}}[3]}\\
\delta_{\widetilde{p},\widetilde{q}}&0&0&0&0&0&\delta_{\widetilde{c},\widetilde{p}}\widetilde{q}&d_{P^{\widetilde{p},\widetilde{q}}_{s(\widetilde{q})}[2]}\\
0&\partial_{\widetilde{\alpha}\widetilde{l}\widetilde{g}}W&0&0&0&0&-\delta_{\widetilde{c},\widetilde{l}}\widetilde{g}^\ast&-\delta_{\widetilde{p},\widetilde{l}}\partial_{\widetilde{q}\widetilde{g}}W&d_{P^{\widetilde{l},\widetilde{g}}_{t(\widetilde{g})}[1]}\\
0&0&0&0&\delta_{\sigma,\tau}&0&\delta_{\widetilde{c},\tau}e_k^\ast&\delta_{\widetilde{p},\tau}\widetilde{q}^\ast&\delta_{\widetilde{l},\tau}\widetilde{g}&d_{P_k^\tau}
\end{smallmatrix}\right)
\]

We have a homomorphisms of dg $\Gamma$-modules $\varphi_i:F(\mathbf{p}\widetilde{S}_i)\to \mathbf{p}S_i^\sharp$ as follows.
\[
\varphi_i=\left(\begin{smallmatrix}
1&0&0&0&0&0&0&0&0&0\\
0&\delta_{\alpha,\widetilde{\alpha}}&0&0&0&0&0&0&0&0\\
0&0&0&0&0&0&0&-\delta_{h,\widetilde{h}}&0&0\\
0&0&0&0&0&\delta_{\beta,\widetilde{\beta}}&0&0&-\partial_{xy\widetilde{\beta}}W&0\\
0&0&0&-\delta_{c,\sigma}e_k^\ast&-\delta_{p,\sigma}q^\ast&0&\delta_{l,\sigma}g&0&0&0\\
0&0&0&0&0&0&0&0&0&1\\
0&0&0&\delta_{c,\widetilde{c}}&0&0&0&0&0&0\\
0&0&0&0&\delta_{p,\widetilde{p}}\delta_{q,\widetilde{q}}&0&0&0&0&0\\
0&0&0&0&0&0&-\delta_{l,\widetilde{l}}\delta_{g,\widetilde{g}}&0&0&0\\
0&0&0&0&0&0&0&0&0&0
\end{smallmatrix}\right)
\]

Similarly, for $i=k$, $F(\mathbf{p}\widetilde{S}_k)$ has the underlying graded space
\[P_k[4]
\oplus\bigoplus\limits_{\begin{smallmatrix}
	\rho\in Q_1\\t(\rho)=k
	\end{smallmatrix}}P_{s(\rho)}^\rho[3]
\oplus\bigoplus\limits_{\begin{smallmatrix}
	\gamma\in Q_1\\s(\gamma)=k
	\end{smallmatrix}}P_{t(\gamma)}^\gamma[2]
\oplus\bigoplus\limits_{\begin{smallmatrix}
	w\in Q_1\\t(w)=k
	\end{smallmatrix}}P^w_{s(w)}[1]
\oplus P_k[1]
\oplus\bigoplus\limits_{\begin{smallmatrix}
	z\in Q_1\\ t(z)=k
	\end{smallmatrix}}P_{s(z)}^z
\]
with the differential
\[\left(\begin{smallmatrix}
d_{P_k[4]}\\
-\rho&d_{P^\rho_{s(\rho)}[3]}\\
-\gamma^\ast&-\partial_{\rho\gamma}W&d_{P^\gamma_{t(\gamma)}[2]}\\
w e_k^\ast&w\rho^\ast&-w\gamma&d_{P^w_{s(w)}[1]}\\
-e_k^\ast&-\rho^\ast&\gamma&0&d_{P_k[1]}\\
0&0&0&\delta_{w,z}&z&d_{P^z_{s(z)}}
\end{smallmatrix}\right)
\]
Then there is a homomorphism of dg $\Gamma$-modules $\varphi_k:F(\mathbf{p}\widetilde{S}_k)\to \mathbf{p}S^\sharp_k$, where $\mathbf{p}S^\sharp_k=\mathbf{p}S_k[1]$ has the underlying graded space
\[P_k[4]
\oplus\bigoplus\limits_{\begin{smallmatrix}
	\widetilde{\rho}\in Q_1\\t(\widetilde{\rho})=k
	\end{smallmatrix}}P_{s(\widetilde{\rho})}^{\widetilde{\rho}}[3]
\oplus\bigoplus\limits_{\begin{smallmatrix}
	\widetilde{\gamma}\in Q_1\\s(\widetilde{\gamma})=k
	\end{smallmatrix}}P_{t(\widetilde{\gamma})}^{\widetilde{\gamma}}[2]
\oplus P_k[1]
\]
with the differential
\[\left(\begin{smallmatrix}
d_{P_k[4]}\\
-\widetilde{\rho}&d_{P^{\widetilde{\rho}}_{s(\widetilde{\rho})}[3]}\\
\widetilde{\gamma}^\ast&\partial_{\widetilde{\rho}\widetilde{\gamma}}W&d_{P^{\widetilde{\gamma}}_{t(\widetilde{\gamma})}[2]}\\
-e_k^\ast&-\widetilde{\rho}^\ast&-\widetilde{\gamma}&d_{P_k[1]}
\end{smallmatrix}\right)
\]
and the homomorphism
\[\varphi_k=\left(\begin{smallmatrix}
1&0&0&0&0&0\\
0&\delta_{\rho,\widetilde{\rho}}&0&0&0&0\\
0&0&-\delta_{\gamma,\widetilde{\gamma}}&0&0&0\\
0&0&0&0&1&0
\end{smallmatrix}\right)\]

It is straightforward to check that the above $\varphi_i$, $i\in Q_0$, are quasi-isomorphisms. Hence we have the following result.

\begin{lemma}
There are isomorphisms in $\D(\Gamma)$:
\[F(\widetilde{S}_i)\xrightarrow{\varphi_i} S_i^\sharp.\]
\end{lemma}

In this subsection, we will further describe the image of morphisms between simples under $F$.


\begin{construction}
For any arrow $\mathfrak{a}:i\to k\in Q_1$ and any arrow $\mathfrak{b}:k\to j\in Q_1$, define
\begin{itemize}
	\item $\pi^\sharp_{\mathfrak{a}'}:S^\sharp_k\to S^\sharp_i[1]$ to be the morphism from $S_k[1]$ to $S^\sharp_i[1]$ given by the identity from $S_k$ to $S^{\mathfrak{a}}_{t(\mathfrak{a})}$;
	\item $\pi^\sharp_{\mathfrak{a}'^\ast}:S^\sharp_i\to S^\sharp_k[2]$ to be the morphism from $S^\sharp_i$ to $S_k[3]$ given by $\pi_{e^\ast_k}:S^{\mathfrak{a}}_{t(\mathfrak{a})}\to S_k[3]$;
	\item $\pi^\sharp_{\mathfrak{b}'}:S_j^\sharp\to S_k^\sharp[1]$ to be $\pi_{b^\ast}:S_j\to S_k[2]$;
	\item $\pi^\sharp_{\mathfrak{b}'^\ast}:S_k^\sharp\to S_j^\sharp[2]$ to be $\pi_{b}[1]:S_k[1]\to S_j[2]$;
	\item $\pi^\sharp_{[\mathfrak{a}\mathfrak{b}]}:S^\sharp_i\to S^\sharp_j[1]$ to be the morphism from $S^\sharp_i$ to $S_j[1]$ given by $\pi_{\mathfrak{b}}:S^{\mathfrak{a}}_{t(\mathfrak{a})}\to S_j[1]$;
	\item $\pi^\sharp_{[\mathfrak{a}\mathfrak{b}]^\ast}:S^\sharp_j\to S^\sharp_i[2]$ to be the morphism from $S_j$ to $S^\sharp_i[2]$ given by $\pi_{\mathfrak{b}^\ast}:S_j\to S^\mathfrak{a}_{t(\mathfrak{a})}[2]$.
\end{itemize}
For any other arrows $\mathfrak{c}$ of $Q$, $\pi^\sharp_{\mathfrak{c}}$ and $\pi^\sharp_{\mathfrak{c}^\ast}$ are given by $\pi_{\mathfrak{c}}$ and $\pi_{\mathfrak{c}^\ast}$, respectively.
\end{construction}

\begin{proposition}\label{prop:simple}
For any arrow $R:s\to t\in \widetilde{Q'}_1$, we have the following commutative diagrams
\[\xymatrix{
	F(S'_s)\ar[d]_{\varphi_s}\ar[r]^{F(\pi_R)}& F(S'_t[1])\ar[d]^{\varphi_t[1]}\\
	S^\sharp_s\ar[r]^{\pi^\sharp_{R}}& S^\sharp_t[1]
}\qquad
\xymatrix{
	F(S_t)\ar[d]_{\varphi_t}\ar[r]^{F(\pi_{R^\ast})}& F(S_s[2])\ar[d]^{\varphi_s[2]}\\
	S^\sharp_t\ar[r]^{\pi^\sharp_{R^\ast}}& S^\sharp_s[2]
}\]
\end{proposition}

\begin{proof}
We lift the homomorphisms in the diagrams between simples to homomorphisms between their cofibrant resolutions. Then we only need to show that the difference of the two compositions in one diagram is null-homotopic.
\begin{enumerate}
	\item The case $R=\mathfrak{a}'$ for some $\mathfrak{a}:i\to k\in Q_1$. By definition, the morphism $F(\pi_{\mathfrak{a}'})$ is given by the map $\mathfrak{p}\pi_{\mathfrak{a}'}:F(\mathbf{p}S'_k)\to F(\mathbf{p}S'_i[1])$, where
	\[
	\mathfrak{p}\pi_{\mathfrak{a}'}=\left(\begin{smallmatrix}
	0&0&0&0&0&0\\
	0&0&0&0&0&0\\
	0&0&0&0&0&0\\
	\delta_{c,\mathfrak{a}}&0&0&0&0&0\\
	0&\delta_{p,\mathfrak{a}}\delta_{\rho,q}&0&0&0&0\\
	0&0&0&0&0&0\\
	0&0&\delta_{l,\mathfrak{a}}\delta_{\gamma,g}&0&0&0\\
	0&0&0&0&0&0\\
	0&0&0&0&0&0\\
	0&0&0&\delta_{w,\mathfrak{a}}&0&0\\
	\end{smallmatrix}\right)
	\]
	So
	\[\varphi_i\circ \mathfrak{p}\pi_{\mathfrak{a}'}=
	\left(\begin{smallmatrix}
	0&0&0&0&0&0\\
	0&0&0&0&0&0\\
	0&0&0&0&0&0\\
	0&0&0&0&0&0\\
	-\delta_{\mathfrak{a},\sigma}e_k^\ast & -\delta_{\mathfrak{a},\sigma}\rho^\ast & \delta_{\mathfrak{a},\sigma}\gamma&0&0&0\\
	0&0&0&\delta_{\mathfrak{a},w}&0&0\\
	\delta_{\mathfrak{a},\widetilde{c}}&0&0&0&0&0\\
	0&\delta_{\mathfrak{a},\widetilde{p}}\delta_{\rho,\widetilde{q}}&0&0&0&0\\
	0&0&-\delta_{\mathfrak{a},\widetilde{l}}\delta_{\gamma,\widetilde{g}}&0&0&0\\
	0&0&0&0&0&0
	\end{smallmatrix}\right)
	\]
	and
	\[
	\pi^\sharp_{\mathfrak{a}'}\circ\varphi_k=
	\left(\begin{smallmatrix}
	0&0&0&0&0&0\\
	0&0&0&0&0&0\\
	0&0&0&0&0&0\\
	0&0&0&0&0&0\\
	0&0&0 &0&0&0\\
	0&0&0&0&0&0\\
	\delta_{\mathfrak{a},\widetilde{c}}&0&0&0&0&0\\
	0&\delta_{\mathfrak{a},\widetilde{p}}\delta_{\rho,\widetilde{q}}&0&0&0&0\\
	0&0&-\delta_{\mathfrak{a},\widetilde{l}}\delta_{\gamma,\widetilde{g}}&0&0&0\\
	0&0&0&0&\delta_{\mathfrak{a},\tau}&0
	\end{smallmatrix}\right)
	\]
	Then the difference $\varphi_i\circ\mathfrak{p}\pi_{\mathfrak{a}'}-\pi^\sharp_{\mathfrak{a}'}\circ\varphi_k=\theta\circ d+d\circ\theta$, where
	\[\theta=\left(\begin{smallmatrix}
	0&0&0&0&0&0\\
	0&0&0&0&0&0\\
	0&0&0&0&0&0\\
	0&0&0&0&0&0\\
	0&0&0&0&\delta_{\mathfrak{a},\sigma}&0\\
	0&0&0&0&0&\delta_{\mathfrak{a},z}\\
	0&0&0&0&0&0\\
	0&0&0&0&0&0\\
	0&0&0&0&0&0\\
	0&0&0&0&0&0
	\end{smallmatrix}\right)
	\]
	of degree -1, and hence this difference is null-homotopic.
	
	For the second diagram, note that
	\[F(\mathbf{p}\pi_{\mathfrak{a}'^\ast})=\left(\begin{smallmatrix}
	0&0&0&0&0&0&0&0&0&0\\
	0&0&0&0&0&0&0&0&0&0\\
	0&0&0&0&0&0&0&0&0&0\\
	\delta_{\mathfrak{a},w}&0&0&0&0&0&0&0&0&0\\
	0&0&0&\delta_{\mathfrak{a},c}&0&0&0&0&0&0\\
	0&0&0&0&\delta_{\mathfrak{a},p}\delta_{q,z}&0&0&0&0&0
	\end{smallmatrix}\right)
	\]
	and
	\[\mathbf{p}\pi^\sharp_{\mathfrak{a}'^\ast}=\left(\begin{smallmatrix}
	0&0&0&0&0&0&0&0&0&0\\
	0&0&0&0&0&0&0&0&0&0\\
	0&0&0&0&0&0&0&0&0&0\\
	0&0&0&0&0&0&\delta_{\mathfrak{a},\widetilde{c}}&0&0&0
	\end{smallmatrix}\right)
	\]
	So
	\[\varphi_k[2]\circ F(\mathbf{p}\pi_{\mathfrak{a}'^\ast})=\left(\begin{smallmatrix}
	0&0&0&0&0&0&0&0&0&0\\
	0&0&0&0&0&0&0&0&0&0\\
	0&0&0&0&0&0&0&0&0&0\\
	0&0&0&\delta_{\mathfrak{a},c}&0&0&0&0&0&0
	\end{smallmatrix}\right)
	\]
	and
	\[\mathbf{p}\pi^\sharp_{\mathfrak{a}'^\ast}\circ\varphi_i=\left(\begin{smallmatrix}
	0&0&0&0&0&0&0&0&0&0\\
	0&0&0&0&0&0&0&0&0&0\\
	0&0&0&0&0&0&0&0&0&0\\
	0&0&0&\delta_{\mathfrak{a},c}&0&0&0&0&0&0
	\end{smallmatrix}\right)
	\]
	Then the difference is zero.
	\item The case $R=\mathfrak{b}'$ for some $\mathfrak{b}:k\to i\in Q_1$.
	Note that
	\[\mathbf{p}F(\pi_{\mathfrak{b}'})=
	\left(\begin{smallmatrix}
	0&0&0&0&0&0&0&0&0&0\\
	0&0&0&0&0&0&0&0&0&0\\
	\delta_{\mathfrak{b},\gamma}&0&0&0&0&0&0&0&0&0\\
	0&0&\delta_{\mathfrak{b},b}\delta_{a,w}&0&0&0&0&0&0&0\\
	0&0&0&0&0&0&0&\delta_{\mathfrak{b},h}&0&0\\
	0&0&0&0&0&0&0&0&\delta_{\mathfrak{b},x}\delta_{y,z}&0
	\end{smallmatrix}\right)\]
	and
	\[\mathbf{p}\pi^\sharp_{\mathfrak{b}'}=
	\left(\begin{smallmatrix}
	0&0&0&0&0&0&0&0&0&0\\
	0&0&0&0&0&0&0&0&0&0\\
	\delta_{\mathfrak{b},\widetilde{\gamma}}&0&0&0&0&0&0&0&0&0\\
	0&0&\delta_{\mathfrak{b},\widetilde{h}}&0&0&0&0&0&0&0
	\end{smallmatrix}\right)
	\]
	So
	\[\mathbf{p}\varphi_k[1]\circ\mathbf{p}F(\pi_{\mathfrak{b}'})=\left(\begin{smallmatrix}
	0&0&0&0&0&0&0&0&0&0\\
	0&0&0&0&0&0&0&0&0&0\\
	-\delta_{\mathfrak{b},\widetilde{\gamma}}&0&0&0&0&0&0&0&0&0\\
	0&0&0&0&0&0&\delta_{\mathfrak{b},h}&0&0&0
	\end{smallmatrix}\right)\]
	and
	\[\mathbf{p}\pi^\sharp_{\mathfrak{b}'}\circ(\varphi_i)=\left(\begin{smallmatrix}
	0&0&0&0&0&0&0&0&0&0\\
	0&0&0&0&0&0&0&0&0&0\\
	-\delta_{\mathfrak{b},\widetilde{\gamma}}&0&0&0&0&0&0&0&0&0\\
	0&0&0&0&0&0&\delta_{\mathfrak{b},h}&0&0&0
	\end{smallmatrix}\right)\]
	Then the difference $\mathbf{p}\varphi_k[1]\circ\mathbf{p}F(\pi_{\mathfrak{b}'})-\mathbf{p}\pi^\sharp_{\mathfrak{b}'}\circ(\varphi_i)=0$.
	
	For the second diagram, note that
	\[F(\mathbf{p}\pi_{\mathfrak{b}'^\ast})=\left(\begin{smallmatrix}
	0&0&0&0&0&0\\
	0&0&0&0&0&0\\
	0&0&0&0&0&0\\
	0&0&0&0&0&0\\
	0&0&0&0&0&0\\
	0&0&0&0&0&0\\
	0&0&0&0&0&0\\
	\delta_{\mathfrak{b},h}&0&0&0&0&0\\
	0&\delta_{\mathfrak{b},x}\delta_{\rho,y}&0&0&0&0\\
	0&0&\delta_{\mathfrak{b},\gamma}&0&0&0
	\end{smallmatrix}\right)
	\]
	and
	\[\mathbf{p}\pi^\sharp_{\mathfrak{b}'^\ast}=\left(\begin{smallmatrix}
	0&0&0&0\\
	0&0&0&0\\
	\delta_{\mathfrak{b},\widetilde{h}}&0&0&0\\
	0&\partial_{\widetilde{\rho}\mathfrak{b}\widetilde{\beta}}W&0&0\\
	0&0&0&0\\
	0&0&\delta_{\mathfrak{b},\widetilde{\gamma}}0&\\
	0&0&0&0\\
	0&0&0&0\\
	0&0&0&0\\
	0&0&0&0
	\end{smallmatrix}\right)
	\]
	So
	\[(\varphi_i[2])\circ F(\mathbf{p}\pi_{\mathfrak{b}'^\ast})=
	\left(\begin{smallmatrix}
	0&0&0&0&0&0\\
	0&0&0&0&0&0\\
	\delta_{\mathfrak{b},\widetilde{h}}&0&0&0&0&0\\
	0&\partial_{\rho\mathfrak{b}\widetilde{\beta}}W&0&0&0&0\\
	0&0&0&0&0&0\\
	0&0&-\delta_{\mathfrak{b},\gamma}&0&0&0\\
	0&0&0&0&0&0\\
	0&0&0&0&0&0\\
	0&0&0&0&0&0\\
	0&0&0&0&0&0
	\end{smallmatrix}\right)
	\]
	and
	\[\mathbf{p}\pi^\sharp_{\mathfrak{b}'^\ast}\circ\varphi_k=
	\left(\begin{smallmatrix}
	0&0&0&0&0&0\\
	0&0&0&0&0&0\\
	\delta_{\mathfrak{b},\widetilde{h}}&0&0&0&0&0\\
	0&\partial_{\rho\mathfrak{b}\widetilde{\beta}}W&0&0&0&0\\
	0&0&0&0&0&0\\
	0&0&-\delta_{\mathfrak{b},\gamma}&0&0&0\\
	0&0&0&0&0&0\\
	0&0&0&0&0&0\\
	0&0&0&0&0&0\\
	0&0&0&0&0&0
	\end{smallmatrix}\right)
	\]
	Then the difference is zero.
	\item It is more straightforward to calculate the cases $R=[\mathfrak{a}\mathfrak{b}]$, $R\in Q_1$.
\end{enumerate}
\end{proof}

\subsection{Geometric interpretation of Keller-Yang's equivalence}

Let $\surfo$ be a decorated marked surface. Denote by $\cA(\surfo)$ the set of simple closed arcs in $\surfo$. Here, closed arcs mean curves in $\surfo-\Tri$ connecting different decorating points. For an arc $\gamma$ in
a triangulation $\TT\in\EGp(\surfo)$, its dual (w.r.t. $\TT$) is the unique closed arc (up to homotopy) which intersects $\gamma$ once and does not cross any other arcs in $\TT$. The dual of $\TT$, denote by $\TT^*$, is defined to be the set of duals of arcs in $\TT$.

For any oriented closed arc $\eta\in\cA(\surfo)$, there is an associated object $X_\eta=X^\TT_{\eta}$ in $\Sph(\Gamma_\TT)$ constructed in \cite[Construction~A.3]{QZ2}.

\begin{proposition}[{\cite[Theorem~6.6]{QQ}, \cite[Proposition~4.3]{QZ2}}]\label{pp:simple}
	There is a canonical bijection
	\begin{gather}\label{eq:X}
	\widetilde{X}_\TT\colon\cA(\surfo)\to\Sph(\Gamma_\TT)/[1]
	\end{gather}
	sending $\eta$ to $X^\TT_\eta[\mathbb{Z}]$.
	Moreover, this is compatible with the isomorphism \eqref{eq:cong} in the following sense.
	Suppose that $\TT'$ is a triangulation in $\EGp(\surfo)$ with its dual $\TT'^*$. Then $\widetilde{X}_\TT(\TT'^\ast)$ is the set of the shift orbits of simples in $\h_\TT^{\TT'}$.
\end{proposition}

Further, for any two oriented closed arcs $\eta_1,\eta_2\in\cA(\surfo)$ having the same starting point $Z$, the oriented angle $\theta$ (in clockwise direction) at $Z$ from $\eta_1$ to $\eta_2$ induces a morphism
$$
    \varphi^\TT(\eta_1,\eta_2)=\varphi(\eta_1,\eta_2):X_{\eta_1}\to X_{\eta_2},
$$
see \cite[Construction~A.5]{QZ2}. We have two useful lemmas.

\begin{lemma}\cite[Corollary~A.9 and Lemma~3.3]{QZ2}\label{lem:A11} Let ${\eta_i}$, for $i=1,2,3$, be oriented closed arcs, which have the same starting point $Z$ and whose start segments are in clockwise order at $Z$. Then
\[\varphi(\eta_2,\eta_3)\circ\varphi(\eta_1,\eta_2)=\varphi(\eta_1,\eta_3).\]
Moreover, this is the only way such that the composition of two $\varphi(-,-)$'s is not zero.
\end{lemma}

\begin{lemma}[{\cite[Proposition~3.1 and Theorem~4.5]{QZ2}}]\label{lem:b}
For any two closed arcs $\eta_1,\eta_2$, if they do not cross each other in $\surf-\Tri$, then the morphisms from $X_{\eta_1}$ to $X_{\eta_2}$ of the form $\varphi(-,-)$ form a basis of $\Ext^\mathbb{Z}(X_{\eta_1},X_{\eta_2})$.
\end{lemma}

We denote by $S_{\eta}$, $\eta\in\TT'^\ast$, the simples in $\h_\TT^{\TT'}$. Then by Proposition~\ref{prop:simple}, we have that $S_\eta\in X^\TT_\eta[\mathbb{Z}]$.
Then by Lemma~\ref{lem:b}, we have

\begin{lemma}\label{lem:con}
$\{\varphi(-,-)\}$ form a basis of the Ext-algebra
\[\E(\h_\TT^{\TT'}):=\Ext^{\mathbb{Z}}\left(\bigoplus\limits_{\eta\in\TT'^\ast}S_\eta,\bigoplus\limits_{\eta\in\TT'^\ast}S_\eta\right).\]
And the multiplication between this basis is given by Lemma~\ref{lem:A11}.
\end{lemma}

It follows directly from the construction of $\varphi(-,-)$ and Proposition~\ref{prop:simple} that this basis gives a nice geometric model for Keller-Yang's equivalence.


\begin{proposition}\label{prop:KY}
For any two closed arcs $\eta_i$ and $\eta_j$ in $\TT'^\ast$, which have the same starting point, the image of $\varphi^{\TT'}(\eta_i,\eta_j)$ under the Keller-Yang's equivalence $\kappa_{\TT'}^{\TT}$ is $\varphi^{\TT}(\eta_i,\eta_j)$.
\end{proposition}

\section{Intrinsic derived equivalences}\label{sec:KY}

In this section, we will first construct an intrinsic equivalence between the finite dimensional derived categories associated to two triangulations (Construction~\ref{con:iota}). Then we show that this equivalence is naturally isomorphic to the composition of any sequence of Keller-Yang's equivalences which connects these two triangulations (Theorem~\ref{thm:comp}). This gives a proof of Theorem~\ref{thma}.

\subsection{The construction}\label{sec:con}


Fix a triangulation $\TT_0$ in $\EGp(\surfo)$ and let $\Gamma_0=\Gamma_{\TT_0}$. Let $\TT$ be any triangulation in $\EGp(\surf)$. Recall that $\h_0^\TT:=\h_{\TT_0}^\TT$ is the heart in $\D_{fd}(\Gamma_0)$ corresponding to $\TT$, and $\h_\TT$ is the canonical heart in $\D_{fd}(\Gamma_\TT)$.

\begin{construction}\label{con:iota}
By Lemma~\ref{lem:con}, there is an isomorphism between Ext algebras
\[
    \iota_\TT\colon\E(\h_0^\TT)\xrightarrow{\sim}\E(\h_\TT),
\]
which sends $\varphi^{\TT_0}(\eta_1,\eta_2)$ to $\varphi^{\TT}(\eta_1,\eta_2)$ for any $\eta_1,\eta_2\in\TT^\ast$.

\end{construction}

As a result, we have an induced triangle equivalence $\Psi_\TT$ fitting the following commutative diagram of equivalences
\begin{gather}\label{eq:deeq}
\xymatrix{
\D_{fd}(\Gamma_0)\ar@/_2pc/[rrrrrr]_{\Psi_\TT}\ar[rr]^{\Ext_{\Gamma_0}^\mathbb{Z}(S_{\TT_0}^{\TT},-)}&&\per\E(\h_0^\TT)\ar[rr]^{\quad\iota_\TT\quad}&& \per\E(\h_\TT)&&\D_{fd}(\Gamma_\TT)\ar[ll]_{\Ext_{\Gamma_\TT}^\mathbb{Z}(S_{\TT}^{\TT},-)}
}
\end{gather}

Consider a sequence of forward/backward flips
$$p\colon\TT_0\xrightarrow{}\TT_1\xrightarrow{}\cdots
    \xrightarrow{}\TT_m=\TT$$
and the sequence of the associated KY's equivalences
$$\D(\Gamma_{\TT_{0}})\xrightarrow{\;\kappa_{\TT_0}^{\TT_1}\;}\D(\Gamma_{\TT_{1}})
    \xrightarrow{\;\kappa_{\TT_1}^{\TT_2}\;}\cdots\xrightarrow{\;\kappa_{\TT_{m-1}}^{\TT_m}\;}\D(\Gamma_{\TT_{m}})=\D(\Gamma_\TT).$$
Restricted to $\D_{fd}$, we obtain a triangle equivalence
\begin{gather}\label{eq:deeq0}
    \Psi(p)=\kappa_{\TT_{m-1}}^{\TT_m}\circ\cdots\circ\kappa_{\TT_1}^{\TT_2}\circ\kappa_{\TT_0}^{\TT_1}\colon \D_{fd}(\Gamma_0) \xrightarrow{\quad\simeq\quad} \D_{fd}(\Gamma_\TT).
\end{gather}

\begin{theorem}\label{thm:comp}
$\Psi_\TT$ and $\Psi(p)$ are naturally isomorphic to each other (denoted by $\Psi_\TT\sim\Psi(p)$),
for any $\TT\in\EGp(\surfo)$ and any sequence of flips $p\colon\TT_0\to\TT$.
\end{theorem}

The remaining of this section is devoted to the proof of this theorem. As a result, we can denote the 3-CY category associated to $\surfo$
by $\D_{fd}(\surfo)$.

\subsection{Compatibility/Proof of Theorem~\ref{thm:comp}}\label{sec:ind}
Use induction on the number $m$ of flips in the flip sequence $p$, starting with
the trivial case, when $m=0$ or $\TT_0=\TT$ so that both equivalences are isomorphic to the identity.
Now suppose that $\Psi_\TT\sim\Psi(p)$ for some $p$ and consider a flip
$\mu_k\colon\TT\to\TT'$ and the flip sequence $p'=\mu_k\circ p$.
Without loss of generality, assume $\mu_k$ is a forward flip.
Fix/recall the notations as follows:
\begin{itemize}
\item $\TT=\{\gamma_i\},\TT^*=\{\eta_i\}$ and $\TT'=\{\gamma_i'\},(\TT')^*=\{\eta_i'\}$.
Note that $\gamma_i'=\gamma_i$ for $i\neq k$.
The local pictures of $\TT$ and $\TT'$ are shown in Figure~\ref{fig:WH}
and the local mutation of the corresponding quiver is:
\begin{gather}\label{eq:mutation}
\xymatrix@C=2.3pc@R=2pc{
    & 2 \ar[d]^{c} &&Q_\TT \ar@{=>}[r]^{\mu_k}&Q_{\TT'}&& 2 \ar@{<-}[d]_{c'}\\
    1\ar[ur]^{b} &k\ar[l]^a \ar[r]^e &3\ar[dl]^f   &&&
    1\ar@{<-}[dr]_{[ag]} &k\ar@{<-}[l]_{a'} \ar@{<-}[r]_{e'}&3\ar@{<-}[ul]_{[ec]}\\
    &4\ar[u]^g     &&&&& 4\ar@{<-}[u]_{g'}
}.\end{gather}
Note that $\eta_i$ might not exist for any $1\leq i\leq 4$ and some vertices might coincide.
\begin{figure}[ht]\centering
\begin{tikzpicture}[xscale=-.4,yscale=.425]
    \path (4,3) coordinate (v2)
          (-4,3) coordinate (v1)
          (2,0) coordinate (v3)
          (0,5) coordinate (v4);
  \draw[blue!30!green!30, dashed,very thin] plot [smooth,tension=0] coordinates {(v1)(v3)(v2)};
  \draw[blue!30!green!30, dashed,very thin] plot [smooth,tension=0] coordinates {(v1)(v4)(v2)};
  \foreach \j in {.1, .18, .26, .34, .42, .5,.58, .66, .74, .82, .9}
    {
      \path (v3)--(v4) coordinate[pos=\j] (m0);
      \draw[blue!30!green!30, dashed,very thin] plot [smooth,tension=.3] coordinates {(v1)(m0)(v2)};
    }
    \path (4,-3) coordinate (v1)
          (-4,-3) coordinate (v2)
          (-2,0) coordinate (v3)
          (0,-5) coordinate (v4);
  \draw[blue!30!green!30, dashed,very thin] plot [smooth,tension=0] coordinates {(v1)(v3)(v2)};
  \draw[blue!30!green!30, dashed,very thin] plot [smooth,tension=0] coordinates {(v1)(v4)(v2)};
  \foreach \j in {.1, .18, .26, .34, .42, .5,.58, .66, .74, .82, .9}
    {
      \path (v3)--(v4) coordinate[pos=\j] (m0);
      \draw[blue!30!green!30, dashed,very thin] plot [smooth,tension=.3] coordinates {(v1)(m0)(v2)};
    }
    \path (4,-3) coordinate (v2)
          (-4,3) coordinate (v1)
          (2,0) coordinate (v3)
          (-2,0) coordinate (v4);
  \draw[blue!30!green!30, dashed,very thin] plot [smooth,tension=0] coordinates {(v1)(v3)(v2)};
  \draw[blue!30!green!30, dashed,very thin] plot [smooth,tension=0] coordinates {(v1)(v4)(v2)};
  \foreach \j in {.13,.26,.39,.87,.74,.61}
    {
      \path (v3)--(v4) coordinate[pos=\j] (m0);
      \draw[blue!30!green!30, dashed,very thin] plot [smooth,tension=.3] coordinates {(v1)(m0)(v2)};
    }
    \path (4,3) coordinate (v1)
          (4,-3) coordinate (v2)
          (2,0) coordinate (v3)
          (6,0) coordinate (v4);
  \draw[blue!30!green!30, dashed,very thin] plot [smooth,tension=0] coordinates {(v1)(v3)(v2)};
  \draw[blue!30!green!30, dashed,very thin] plot [smooth,tension=0] coordinates {(v1)(v4)(v2)};
  \foreach \j in {.1,.2,.3,.4,.5,.6,.7,.8,.9}
    {
      \path (v3)--(v4) coordinate[pos=\j] (m0);
      \draw[blue!30!green!30, dashed,very thin] plot [smooth,tension=.3] coordinates {(v1)(m0)(v2)};
    }
    \path (-4,-3) coordinate (v2)
          (-4,3) coordinate (v1)
          (-2,0) coordinate (v3)
          (-6,0) coordinate (v4);
  \draw[blue!30!green!30, dashed,very thin] plot [smooth,tension=0] coordinates {(v1)(v3)(v2)};
  \draw[blue!30!green!30, dashed,very thin] plot [smooth,tension=0] coordinates {(v1)(v4)(v2)};
  \foreach \j in {.1,.2,.3,.4,.5,.6,.7,.8,.9}
    {
      \path (v3)--(v4) coordinate[pos=\j] (m0);
      \draw[blue!30!green!30, dashed,very thin] plot [smooth,tension=.3] coordinates {(v1)(m0)(v2)};
    }
\draw[NavyBlue,thin]
  (4,-3)node{$\bullet$}to(4,3)node{$\bullet$}to(-4,3)node{$\bullet$}to(-4,-3)node{$\bullet$}to(4,-3)to(-4,3);
\draw[red,ultra thick](2,0)to(-2,0);
\draw[red,thick]
  (-6,0)node{$\bullet$}node[white]{\tiny{$\bullet$}}node{\tiny{$\circ$}}to
  (-2,0)node{$\bullet$}node[white]{\tiny{$\bullet$}}node{\tiny{$\circ$}}node[above]{$_Y$}to
  (2,0)node{$\bullet$}node[white]{\tiny{$\bullet$}}node{\tiny{$\circ$}}node[below]{$_Z$}to
  (6,0)node{$\bullet$}node[white]{\tiny{$\bullet$}}node{\tiny{$\circ$}}
  (0,5)node{$\bullet$}node[white]{\tiny{$\bullet$}}node{\tiny{$\circ$}}to(2,0)
  (0,-5)node{$\bullet$}node[white]{\tiny{$\bullet$}}node{\tiny{$\circ$}}to(-2,0)
  (0,0)node[above]{$\eta_k$}(4.5,0)node[above]{$\eta_1$}(-4.5,0)node[above]{$\eta_3$}
  (1.5,3)node[above]{$\eta_2$}(-1.5,-3)node[below]{$\eta_4$};
\end{tikzpicture}\quad
\begin{tikzpicture}[scale=1.3, rotate=0]
\draw[blue,->,>=stealth](3-.6,0)to(3+.6,0);\draw (3,-1.8);
\end{tikzpicture}\quad
\begin{tikzpicture}[xscale=.4,yscale=.425]
    \path (4,3) coordinate (v1)
          (-4,3) coordinate (v2)
          (2,0) coordinate (v3)
          (0,5) coordinate (v4);
  \draw[blue!30!green!30, dashed,very thin] plot [smooth,tension=0] coordinates {(v1)(v3)(v2)};
  \draw[blue!30!green!30, dashed,very thin] plot [smooth,tension=0] coordinates {(v1)(v4)(v2)};
  \foreach \j in {.1, .18, .26, .34, .42, .5,.58, .66, .74, .82, .9}
    {
      \path (v3)--(v4) coordinate[pos=\j] (m0);
      \draw[blue!30!green!30, dashed,very thin] plot [smooth,tension=.3] coordinates {(v1)(m0)(v2)};
    }
    \path (4,-3) coordinate (v1)
          (-4,-3) coordinate (v2)
          (-2,0) coordinate (v3)
          (0,-5) coordinate (v4);
  \draw[blue!30!green!30, dashed,very thin] plot [smooth,tension=0] coordinates {(v1)(v3)(v2)};
  \draw[blue!30!green!30, dashed,very thin] plot [smooth,tension=0] coordinates {(v1)(v4)(v2)};
  \foreach \j in {.1, .18, .26, .34, .42, .5,.58, .66, .74, .82, .9}
    {
      \path (v3)--(v4) coordinate[pos=\j] (m0);
      \draw[blue!30!green!30, dashed,very thin] plot [smooth,tension=.3] coordinates {(v1)(m0)(v2)};
    }
    \path (4,-3) coordinate (v2)
          (-4,3) coordinate (v1)
          (2,0) coordinate (v3)
          (-2,0) coordinate (v4);
  \draw[blue!30!green!30, dashed,very thin] plot [smooth,tension=0] coordinates {(v1)(v3)(v2)};
  \draw[blue!30!green!30, dashed,very thin] plot [smooth,tension=0] coordinates {(v1)(v4)(v2)};
  \foreach \j in {.13,.26,.39,.87,.74,.61}
    {
      \path (v3)--(v4) coordinate[pos=\j] (m0);
      \draw[blue!30!green!30, dashed,very thin] plot [smooth,tension=.3] coordinates {(v1)(m0)(v2)};
    }
    \path (4,3) coordinate (v1)
          (4,-3) coordinate (v2)
          (2,0) coordinate (v3)
          (6,0) coordinate (v4);
  \draw[blue!30!green!30, dashed,very thin] plot [smooth,tension=0] coordinates {(v1)(v3)(v2)};
  \draw[blue!30!green!30, dashed,very thin] plot [smooth,tension=0] coordinates {(v1)(v4)(v2)};
  \foreach \j in {.1,.2,.3,.4,.5,.6,.7,.8,.9}
    {
      \path (v3)--(v4) coordinate[pos=\j] (m0);
      \draw[blue!30!green!30, dashed,very thin] plot [smooth,tension=.3] coordinates {(v1)(m0)(v2)};
    }
    \path (-4,-3) coordinate (v1)
          (-4,3) coordinate (v2)
          (-2,0) coordinate (v3)
          (-6,0) coordinate (v4);
  \draw[blue!30!green!30, dashed,very thin] plot [smooth,tension=0] coordinates {(v1)(v3)(v2)};
  \draw[blue!30!green!30, dashed,very thin] plot [smooth,tension=0] coordinates {(v1)(v4)(v2)};
  \foreach \j in {.1,.2,.3,.4,.5,.6,.7,.8,.9}
    {
      \path (v3)--(v4) coordinate[pos=\j] (m0);
      \draw[blue!30!green!30, dashed,very thin] plot [smooth,tension=.3] coordinates {(v1)(m0)(v2)};
    }
\draw[NavyBlue,thin]
  (4,-3)node{$\bullet$}to(4,3)node{$\bullet$}to(-4,3)node{$\bullet$}to(-4,-3)node{$\bullet$}to(4,-3)to(-4,3);
\draw[red,ultra thick](2,0)to(-2,0);
\draw[red,thick]
  (-6,0)node{$\bullet$}node[white]{\tiny{$\bullet$}}node{\tiny{$\circ$}}to
  (-2,0)node{$\bullet$}node[white]{\tiny{$\bullet$}}node{\tiny{$\circ$}}node[above]{$_Z$}to
  (2,0)node{$\bullet$}node[white]{\tiny{$\bullet$}}node{\tiny{$\circ$}}node[below]{$_Y$}to
  (6,0)node{$\bullet$}node[white]{\tiny{$\bullet$}}node{\tiny{$\circ$}}
  (0,5)node{$\bullet$}node[white]{\tiny{$\bullet$}}node{\tiny{$\circ$}}to(2,0)
  (0,-5)node{$\bullet$}node[white]{\tiny{$\bullet$}}node{\tiny{$\circ$}}to(-2,0)
  (0,0)node[above]{$\eta_k$}(4.5,0)node[above]{$\eta_3$}(-4.5,0)node[above]{$\eta_1$}
  (1.5,3)node[above]{$\eta'_{2}$}(-1.5,-3)node[below]{$\eta'_{4}$};
\end{tikzpicture}
\caption{A forward flip}
\label{fig:WH}
\end{figure}
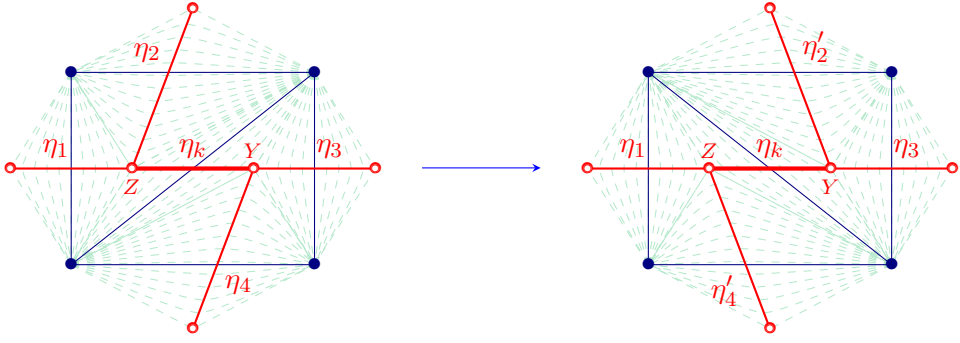
\item $\h_0^\TT$ and $\h_0^{\TT'}$ are the hearts with simples $\{S_{\eta_i}\}$ and $\{S_{\eta_i'}\}$
in $\D_{fd}(\Gamma_0)$ that correspond to $\TT$ and $\TT'$, respectively.
Note that $\h_0^{\TT'}$ is the forward tilting of $\h_0^{\TT}$ w.r.t. $S_{\eta_k}$.

\item $\h_\TT$ and $\h_{\TT'}$ are the canonical hearts with simples $\{S_{i}\}$ and $\{S_i'\}$
in $\D_{fd}(\Gamma_\TT)$ and $\D_{fd}(\Gamma_{\TT'})$, respectively.

\item Moreover, $\h_\TT^{\TT'}$ is the forward tilting of $\h_\TT$ w.r.t. $S_k$,
with simples $\{S_i^\sharp\}$ in $\D_{fd}(\Gamma_\TT)$ (see Construction~\ref{con:sharp} for the construction of $S_i^\sharp$).
\end{itemize}

Note that we shall prove $\Psi_{\TT'}\sim\Psi(\mu_k\circ p)=\kappa_{\TT}^{\TT'}\circ\Psi(p)$, where the latter is $\kappa_{\TT}^{\TT'}\circ\Psi_{\TT}$ by induction. By definition, it suffices to show that $\kappa_{\TT}^{\TT'}\circ\Psi_{\TT}$ induces the isomorphism $\iota_{\TT'}$, which means that $\kappa_{\TT}^{\TT'}\circ\Psi_{\TT}$ preserves the morphism of the from $\varphi^{\TT'}(-,-)$ induced by any angle in $\TT'^\ast$. By Proposition~\ref{prop:KY}, we have that $\kappa_{\TT}^{\TT'}$ preserves such morphisms. So it suffices to show that so does $\Psi_{\TT}$.

If there is no arrow in $Q_{\TT'}$ from $k$ to $i$, that is, $i\neq 2$ or $4$ in Figure~\ref{fig:WH},
then $\eta'_i=\eta_i$.
Hence we only need to consider the angles between $\eta'_2$ and another arc (similarly for $\eta'_4$).  For the angles to $\eta'_2$ in $\TT'^\ast$, we have the following cases, 
up to dual (i.e. the case starting at $\eta'_2$):
\begin{itemize}
\item For an angle at $Y$ from $\eta_k$ to $\eta'_2$, 
by \cite[Proposition~3.1]{QZ2}, we have a triangle
\[S'_{\eta_2}\xrightarrow{\varphi(\eta_2,\eta_k)} S'_{\eta_k}\xrightarrow{\varphi(\eta_k,\eta'_2)} S'_{\eta'_2}\to S'_{\eta_2}[1].\]
As $\iota_\TT$ (and so $\Phi_\TT$) preserves $\varphi(\eta_2,\eta_k)$, 
we deduce that $\Phi_\TT$ preserves this triangle and hence $\varphi(\eta_k,\eta'_2)$. 

Note that when the number of arrows from $k$ to $2$ in $Q_{\TT'}$ is 2, i.e., the vertices $2$ and $4$ coincide, then we need to add another copy of $S'_{\eta_k}$ to the second term of the triangle. 
However the rest of the deduction and conclusion are the same.
\item
For the angle at $Y$ from $\eta_3$ to $\eta'_2$, by Lemma~\ref{lem:A11}, we have $\varphi(\eta_3,\eta'_2)=\varphi(\eta_k,\eta'_2)\circ\varphi(\eta_3,\eta_k)$. 
As above, $\Phi_\TT$ preserves $\varphi(\eta_k,\eta'_2)$ and $\varphi(\eta_3,\eta_k)$. 
Hence it preserves $\varphi(\eta_3,\eta'_2)$ too.
\item 
For any angle to $\eta'_2$ in $\TT'^\ast$, which is at the other endpoint of $\eta'_2$ from $Y$, it factors through $\eta_2$ (i.e. decomposes).
Again, we can prove it is preserved in the same fashion.
\end{itemize}
\section{An application}\label{sec:app}

\subsection{Calabi-Yau categories and spherical objects}\label{sec:DC}
\begin{definition}\label{def:sph}
A triangulated $\k$-category $\D$ is called \emph{$N$-Calabi-Yau} (or $N$-CY for short)
if for any pair of objects $L,M$ in $\D$, we have a natural isomorphism
\begin{equation}\label{eq:serre}
    \Hom_{\D_{fd}(\Gamma)}(L,M)\cong D\Hom_{\D_{fd}(\Gamma)}(M,L[N])
\end{equation}
where $D=\Hom_\k(-,\k)$.
Further, an object $S$ in a $N$-CY triangulated $\k$-category $\D$ is \emph{($N$-)spherical} if $\Hom_{\D}(S, S[i])=\k$ for $i=0$ or $N$, and $0$ otherwise.

The \emph{twist functor $\phi$ of a spherical object} $S$
is defined by
\begin{gather}\label{eq:phi}
    \phi_S(X)=\Cone\left(S\otimes\Hom^\bullet(S,X)\to X\right)
\end{gather}
with inverse
\[
    \phi_S^{-1}(X)=\Cone\left(X\to S\otimes\Hom^\bullet(X,S)^\vee \right)[-1]
\]
\end{definition}

Recall that $\D_{fd}(\Gamma)$ is the finite-dimensional derived category of $\Gamma$,
for a Ginzburg dg algebra $\Gamma$.
It is well-known that this is a 3-CY category.
We also know that $\D_{fd}(\Gamma)$ admits a canonical heart $\zero$ generated
by simple $\Gamma$-modules $S_i$, for $i\in Q_0$, each of which is 3-spherical.
Denote by $\ST(\Gamma)$ the spherical twist group of $\D_{fd}(\Gamma)$
in $\Aut\D_{fd}(\Gamma)$, generated by $\{\phi_{S_i}\mid i\in Q_0\}$.
Further, the set of reachable spherical objects is
\begin{gather}\label{eq:sph=st}
    \Sph(\Gamma)=\ST(\Gamma)\cdot\Sim\zero,
\end{gather}
which is equivalent to the definition in \eqref{eq:sph} (cf. \cite[Lemma~9.2]{QQ}).

For $\D_{fd}(\surfo)$, we will use notation
$\Sph(\surfo)$ and $\ST(\surfo)$ instead.
Furthermore, by \eqref{eq:cong}, we will not distinguish
$\EGp(\surfo)$ and $\EGp(\Gamma_\TT)$.
\subsection{Stability conditions}\label{sec:sc}
Recall the definition of stability conditions as follows.
\begin{definition}[{\cite[Definition~3.3]{B}}]\label{def:stab}
A \emph{stability condition} $\sigma = (Z,\hua{P})$ on $\hua{D}$ consists of
a group homomorphism $Z:K(\hua{D}) \to \kong{C}$ called the \emph{central charge} and
full additive subcategories $\hua{P}(\varphi) \subset \hua{D}$
for each $\varphi \in \kong{R}$, satisfying the following axioms:
\begin{itemize}
\item if $0 \neq E \in \hua{P}(\varphi)$
then $Z(E) = m(E) \exp(\varphi  \pi \mathbf{i} )$ for some $m(E) \in \kong{R}_{>0}$,
\item for all
$\varphi \in \kong{R}$, $\hua{P}(\varphi+1)=\hua{P}(\varphi)[1]$,
\item if $\varphi_1>\varphi_2$ and $A_i \in \hua{P}(\varphi_i)$
then $\Hom_{\hua{D}}(A_1,A_2)=0$,
\item (HN-property) for each nonzero object $E \in \hua{D}$ there is a finite sequence of real numbers
$$\varphi_1 > \varphi_2 > ... > \varphi_m$$
and a collection of triangles
$$\xymatrix@C=0.8pc@R=1.4pc{
  0=E_0 \ar[rr] && E_1 \ar[dl] \ar[rr] &&   E_2 \ar[dl] \ar[rr] && ... \
  \ar[rr] && E_{m-1} \ar[rr] && E_m=E \ar[dl] \\
  & A_1 \ar@{-->}[ul]  && A_2 \ar@{-->}[ul] &&  && && A_m \ar@{-->}[ul]
},$$
with $A_j \in \hua{P}(\varphi_j)$ for all j.
\end{itemize}
\end{definition}

A crucial result about stability condition is that
they form a complex manifold.

\begin{theorem}[Bridgeland\cite{B}]
All stability conditions on a triangulated category
$\D$ form a complex manifold, denoted by $\Stab(\D)$;
each connected component of $\Stab(\D)$ is locally homeomorphic to a linear sub-manifold of
$\Hom_{\kong{Z}}(K(\D),\kong{C})$, sending a stability condition $(\h, Z)$ to its central change $Z$.
\end{theorem}

We will study the pricipal component $\Stap(\surfo)$ of
the space of stability conditions on $\D_{fd}(\surfo)$,
that is the connected component containing stability conditions whose hearts are in
$\EGp(\surfo)$.

\subsection{Faithful actions}\label{sec:ff}
\begin{lemma}
An auto-equivalence $\varphi\in\Aut\D_{fd}(\surfo)$ acts trivially on $\Stap(\surfo)$
if and only if
it acts trivially on $\Sph(\surfo)$.
\end{lemma}
\begin{proof}
As to give a stability condition is equivalent to give a heart $\mathcal{H}$
with a stability function $Z$ on $\mathcal{H}$ satisfying the HN-property in Definition~\ref{def:stab} (see \cite[Proposition~5.3]{B}),
we have the following equivalences
\begin{itemize}
  \item  $\varphi$ acts trivially on $\Stap(\surfo)$;
  \item $\varphi$ acts trivially on the exchange graph $\EGp(\surfo)$;
  \item $\varphi$ acts trivially on any vertices of $\EGp(\surfo)$
  and any edges of $\EGp(\surfo)$.
\end{itemize}
As a heart in $\EGp(\surfo)$ is determined by its simples and the edges of $\EGp(\surfo)$
are labeled by simple of hearts,
we deduce that $\varphi$ acts on $\Stap(\surfo)$ if and only if it acts trivially on the set
\[ \bigcup_{\mathcal{H}\in\EGp(\surfo)}\Sim\mathcal{H}.\]
This is $\Sph(\surfo)$ by \eqref{eq:sph}.
\end{proof}

\begin{theorem}\label{thmbb}
The spherical twist group $\ST(\surfo)$ acts faithfully on $\Stap(\surfo)$.
\end{theorem}
\begin{proof}
Choose any $\h_{\TT}\in\EGp(\surfo)$ that corresponds to a triangulation $\TT$.
Let $\phi\in\ST(\surfo)$. By \cite[Corollary~8.5]{KQ},
$\phi(\h_\TT)$ can be obtained from $\h_\TT$ by a sequence of tiltings.
Hence $\phi(\h_\TT)=\h_{\TT'}$ for some $\TT'$, which is obtained from $\TT$
by the corresponding sequence of flips.
Hence, $\phi$ can be realized as the composition of a sequence of KY equivalences.
By Theorem~\ref{thm:comp}, $\phi$ can be determined by $\h_\TT$ and $\h_{\TT'}$ directly.

In the case that $\phi$ acts trivially on $\Stap(\surfo)$ or $\Sph(\surfo)$,
we have $\phi(\h_{\TT})=\h_{\TT'}$ and the corresponding equivalence from Construction~\ref{con:iota}
is the identity.
Thus, $\phi$ is naturally isomorphic to the identity as required.
\end{proof}

In \cite{QQ}, we have $\ST(\surfo)/\Aut_0\cong\BT(\surfo)$, where $\BT(\surfo)$ is the braid twist group of $\surfo$, and
where $\Aut_0$ is the part of $\Aut\D_{fd}(\surfo)$ that acts trivially on $\Stap\D_{fd}(\surfo)$.
Hence a consequence of the theorem above is the following.

\begin{corollary}
$\ST(\surfo)=\BT(\surfo)$.
\end{corollary}

\end{document}